\documentclass[letterpaper,12pt,oneside]{article}%%%%%%%%%%%%% NAME= GII1
\usepackage[T1]{fontenc}
\usepackage[latin1]{inputenc}
\usepackage{epsfig}
\usepackage{amsmath}
\usepackage{amssymb}
\usepackage{mathtools}
\usepackage{amsthm}
\usepackage[active]{srcltx}
\usepackage{dsfont}
\usepackage[titles]{tocloft}
\usepackage{color}
\numberwithin{equation}{section}
\newtheorem{theoreme}{Theorem}[section]
\newtheorem{proposition}{Proposition}[section]

\newtheorem{claim}[theoreme]{Claim}

\newtheorem{lemme}{Lemma}[section]

\newtheorem{corollaire}[theoreme]{Corollary}
\newtheorem{definition}[theoreme]{Definition}

\definecolor{darkblue}{rgb}{0,0,0.7} % darkblue color
\newcommand{\darkblue}{\color{darkblue}} % darkblue command
\newcommand{\defn}[1]{\emph{\darkblue #1}} % emphasis of a definition

\newcommand{\RR}{\ensuremath{\mathbb R}}

\newcommand{\NN}{\ensuremath{\mathbb N}}
\newcommand{\A}{\ensuremath{\mathcal A}}
\newcommand{\V}{\ensuremath{\mathcal V}}
\newcommand{\U}{\ensuremath{\mathcal U}}
\newcommand{\B}{\ensuremath{\mathcal B}}

\newcommand{\Z}{\ensuremath{\mathcal Z}}

\newcommand{\ind}{\ensuremath{\mathds 1}}

\newcommand{\xx}{\ensuremath{\mathbf{x}}}

\newcommand{\ww}{\ensuremath{\mathbf{w}}}

\newcommand{\dd}{\ensuremath{\mathbf{d}}}
\newcommand{\zz}{\ensuremath{\mathbf{z}}}

\newcommand{\de}{\delta}
\newcommand{\be}{\beta}
\newcommand{\ep}{\varepsilon}
\newcommand{\al}{\alpha}

\newcommand{\ga}{\gamma}
\newcommand{\Ga}{\Gamma}

\newcommand{\vvv}{\ensuremath{\mathbf{v}}}
\newcommand{\VV}{\ensuremath{\mathbf{V}}}

\newcommand{\uuu}{\ensuremath{\mathbf{u}}}

\newcommand{\limn}{\lim_{n\to\infty}}

\newcommand{\vs}{\vspace{-0.25cm}}

 \def\cW{{\mathcal
W}}

\title{A strategy-based proof of the existence of the value in zero-sum differential games }

\author{Juan Pablo Maldonado L\'opez and Miquel Oliu-Barton \footnote {The authors are particularly indebted with Pierre Cardaliaguet, 
Marc Quincampoix and Sylvain Sorin for their careful reading and comments on earlier drafts. This work was partially supported by the Commission of the European
Communities under the 7th Framework Programme Marie Curie Initial Training Network (FP7-PEOPLE-2010-ITN),  project SADCO, contract number 264735. }}
\date{December 10, 2012}

\begin{document}
%\vs \vs \vs \vs \vs \vs
\maketitle 

%\texttt{Version 20}
\vs \vs 
%  \vs \vs \vs \vs \vs \setcounter{tocdepth}{2}
%  \setlength{\cftbeforesubsecskip}{-2ex}
%  \setlength{\cftbeforesecskip}{-2ex} \tableofcontents

\abstract{The value of a zero-sum differential games is known to exist, under Isaacs' condition, as the unique viscosity solution of a Hamilton-Jacobi-Bellman equation.
In this note we provide a self-contained proof based on the construction of $\ep$-optimal strategies, which is inspired by the 
``extremal aiming'' method from \cite{KS87}. }
%, and is in the spirit of proximal normals and approachability. }% A characterization for the value function is obtained} 
%MThis approach provides a (new?) characterization of (1-\la)  the value function.}
% via \textit{extremal aiming}, as introduced by
%Krasovskii and Subbotin \cite{KS87}.
\section{Comparison of trajectories}\label{dir}
 %This construction may be applied tosum
 %construct $\ep$-optimal strategies in differential games with incomplete information, recently introduced by Cardaliaguet in \cite{carda07}, in the case wher. We present here the construction of such strategies, in the case of incomplete information on one side.
\label{prelim}
Let $U$ and $V$ be compact subsets of some euclidean space, let $\|\cdot\|$ be the euclidean norm in $\RR^n$,
% and $d:\RR^n\times \mathcal{P}(\RR^n)\to \RR_+$ the usual distance.
and let $f:[0,1]\times \RR^n \times U \times V\to \RR^n$. Let $\Pi=\{t_0<t_1<\dots < t_N \}$ be a set of times in $[0,1]$,
and let $\|\Pi\|:=\max_{1\leq m\leq N}t_m-t_{m-1}$. For any $\mathcal{Z}\subset \RR^n$, let $D(x,\mathcal{Z}):=\inf_{z\in \mathcal{Z}} \|x-z\|$ be the usual distance to the set $\Z$.\\
%For any closed, nonempty set $F\subset \RR^n$ and any $x\in \RR^n$, let 
%$\pi$be a selection rule that gives a closest point to $x$ in $F$.
%Let $\pi:\RR^n\times F\subset \RR^n \to \RR^n$ a selection rule that assigns, to any $x$ and closed, nonempty set $F$, a closest point to $x$ in $F$.
\textbf{Assumption 1}:\\% (regularity of $f$):\\
\textbf{a.} $f$ is uniformly bounded, i.e. $\|f\|:=\sup_{(t,x,u,v)}\|f(t,x,u,v)\|<+\infty$.\\
\textbf{b.} There exists $c\geq 0$ such that: $\forall (u,v)\in U
\times V,\ \forall s,t \in [0,1],\ \forall x,y \in \RR^n$,
\begin{equation*}
\|f(t,x,u,v)-f(s,y,u,v)\|\leq c\big(|t-s|+\|x-y\|\big).
\end{equation*}
%, m=1,\dots, N\}$. 
%%%%%%%%%%% notation %%%%%%%%
%\textbf{Notation:} $\langle\cdot,\cdot\rangle$, $\|\cdot\|$,  and
%$d$ denote the scalar product, euclidan norm and distance in $\RR^n$. Let $B(x,r)$ denote
%the closed ball of center $x$ and radius $r$,  i.e. $\{y\in \RR^n\ | \ \|y-x\|\leq r \}$.
%For any set $C\subset \RR^n$, $d(x,C)=\inf_{z\in
%C}d(x,z)$. 
%%%%%%%%%%% notation %%%%%%%%
%For any set $\mathcal{C}\subset [0,1]\times\RR^n$, $\mathcal{C}(t)=\{x\in \RR^n | (t,x)\in \mathcal{C}\}$.
%\noindent \textbf{The directional game} %}\label{dir} %\emph{direction} 
\paragraph{The local game:} %}\label{dir} %\emph{direction} 
For any $(t,x)\in [0,1]\times \RR^n$ and any $\xi\in \RR^n$, consider
 the one-shot game with actions sets $U$ and $V$ and
  payoff function: $$(u,v)\mapsto \langle \xi, f(t,x,u,v)\rangle.$$
%\[G_{(t,x,\xi)}(u,v)=\langle \xi, f(t,x,u,v)\rangle.\]
%It is the 
Let this game be denoted by $\Ga(t,x,\xi)$, and let $H^-(t,x,\xi)$ and $H^+(t,x,\xi)$ be its maxmin and minmax respectively: %$\forall (t,x,\xi)\in [t_0,1]\times \RR^n\times \RR^n$,
\begin{eqnarray*}\label{ham}
H^-(t,x,\xi)&:=& \max_{u\in U}\min_{v\in V} % G_{(t,x,\xi)}(u,v),\\
\langle \xi,f(t,x,u,v)\rangle, \\ 
H^+(t,x,\xi)&:=& %\max_{u\in U}\min_{v\in V} G_{(t,x,\xi)}(u,v).
\min_{v\in V}\max_{u\in U}\langle \xi, f(t,x,u,v)\rangle.% \left\{ +\ga(t,x,u,v)\right\}.
\end{eqnarray*}
% Notice that $H^+(t,x,\xi)$ and $H^-(t,x,\xi)$ are, respectively, the upper and lower value of
% the directional game $\Ga(t,x,\xi)$ defined in section \eqref{dir}.
These functions satisfy $H^-\leq H^+$.
If the equality $H^+(t,x,\xi)=H^-(t,x,\xi)$ holds, the game $\Ga(t,x,\xi)$ has a value, denoted by $H(t,x,\xi)$. \\ % $H(t,x,\xi):=H^-(t,x,\xi)=H^+(t,x,\xi)$.\\
%=H^+(t,x,\xi)$.% In thuis case, denote $H(t,x,\xi)$.
%If $H(t,x,\xi)=H^-(t,x,\xi)=H^+(t,x,\xi)$.
\noindent \textbf{Assumption 2}: $\Ga(t,x,\xi)$ has a value for all $(t,x,\xi)\in [0,1]\times \RR^n \times \RR^n$.\\

We suppose that the Assumptions $1$ and $2$ hold in the rest of the paper. % are supposed to hold in the sequel. % in throughout Section \ref{dir}. % we will assume that hold. 
\subsection{A key Lemma} %An important Lemma}%\label{dir}
 Introduce the sets of controls:
\[
\mathcal{U}=\{\uuu:[0,1]\to U, \ \mathrm{ measurable}\},\quad \mathcal{V}=\{\vvv:[0,1]\to V,  \ \mathrm{ measurable}\}.\]
%% \begin{eqnarray*}
% \mathcal{U}&=&\{\uuu:[0,1]\to U, \ \text{ measurable}\},\\ 
% \mathcal{V}&=&\{\vvv:[0,1]\to V,  \ \text{ measurable}\}.
% \end{eqnarray*} 
% For any $t_0\in [0,1]$, introduce the sets of controls:
% \begin{eqnarray*}
% \mathcal{U}(t_0)&=&\{\uuu:[t_0,1]\to U; \uuu \text{ is measurable}\}, \\
% \mathcal{V}(t_0)&=&\{\vvv:[t_0,1]\to V; \vvv \text{ is measurable}\}.
% \end{eqnarray*} 
Elements of $U$ and $V$ will be identified with constant controls. For any $(t_0, z_0,\uuu,\vvv)\in [0,1]\times\RR^n\times\U\times \V$ denote by $\zz[t_0,z_0, \uuu, \vvv]$ the solution of:
\begin{equation*}
\dot{\zz}(t) = f(t,\zz(t),\uuu(t), \vvv(t)), \ \ \zz(t_0) = z_0.
\end{equation*}
\noindent
%evaluated at time $t$. 
Let $(\uuu,\vvv)\in \U\times \V$ be a pair of controls, $t_0\in [0,1]$ an initial time, $(x_0,w_0)\in (\RR^n)^2$ a pair of initial positions, 
and $(u^*,v^*)$ a couple of optimal actions in $\Ga(t_0,x_0,x_0-w_0)$.
Note any pair $(u,v)\in U\times V$ is optimal in this local game if $x_0=w_0$.  %in $\Ga(t_0,x_0,0)$.
 Consider  $\xx(t):=\xx[t_0,x_0,\uuu,v^*](t)$
and $\ww(t):=\ww[t_0,w_0,u^*, \vvv](t)$. The following lemma is inspired by Lemma 2.3.1 in \cite{KS87}.
The existence of the value in the local games will be used to bound the distance between these two trajectories.
%It gives a bound for the distance of these two trajectories over time.

 %Define two continuous trajectories in $\RR^n$, $\xx:[t_0,1]\to \RR^n$ and $\ww:[t_0,1]\to \RR^n$, 
%by: %, i.e. functions in $\mathcal{C}([t_0,1],\RR^n)$ 
%\begin{eqnarray*}
%\xx(t_0)=x_0,& \text{ and } &\dot{\xx}(t)=f(t,\xx(t),\uuu(t),v^*), \text{ a.e.}\\
%\ww(t_0)=w_0,& \text{ and } &\dot{\ww}(t)=f(t,\ww(t),u^*,\vvv(t)), \text{ a.e.}
%\end{eqnarray*}

%from Krasovskii and Subotin \cite{KS87}. %We will use the %Introduce some notation.
%Introduce the will use the following:
\begin{lemme}\label{fundamental}
There exists real numbers $A,B\geq 0$ such that for all $t\in [t_0,1]$:
%\begin{equation}\label{lem}
 %\dd^2(t) \leq (1+(t-t_0)A)d_0^2 + B(t-t_0)^2.
%\end{equation}
\begin{equation*}\label{lem}
 \|\xx(t)-\ww(t)\|^2 \leq (1+(t-t_0)A)\|x_0-w_0\|^2 + B(t-t_0)^2.
\end{equation*}
\end{lemme}
\begin{proof} 
Let $d_0:=\|x_0-w_0\|$ and $\dd(t):=\|\xx(t)-\ww(t)\|$.
Then:  %Put $d_0=\|x_0-w_0\|$ and $\dd(t):=\|\xx(t)-\ww(t)\|$. Then
\begin{eqnarray}\label{pr}
 \dd^2(t)&=&  \|(x_0-w_0)+\int_{t_0}^t f(s,\xx(s),\uuu(s),v^*)-f(s,\ww(s),u^*,\vvv(s))ds\|^2.
 \end{eqnarray}
The boundedness of $f$ implies that: 
\begin{equation}\label{easy} \left \|\int_{t_0}^{t}f(s,\xx(s),\uuu(s),v^*)-f(s,\ww(s),u^*,\vvv(s))ds \right \|^2\leq 4\|f\|^2(t-t_0)^2.
\end{equation}
\textbf{Claim:} For all $s\in [t_0,1]$, and for all $(u,v)\in U \times V$:
%Let us boundConsequently, it is enough to bound:
\begin{equation}\label{seasy} \langle x_0-w_0, f(s,\xx(s),u,v^*)-f(s,\ww(s),u^*,v)ds \rangle\leq 2C(s)d_0+cd_0^2,
\end{equation}
where  $C(s):=c(1+\|f\|)(s-t_0)$.\\
Let us prove this claim.  %for each pair $(u,v)$.
%Let $(u,v)$ and $s\in [t_0,1]$. 
Assumption $1$ implies $\|\xx(s)-x_0\|\leq (s-t_0)\|f\|$, and then: % ; Assumption $(1b)$ gives then: %. The Lipschitz property implies then: %n, the Lipschitz property , and also that:%. The Lipschitz assumption on $f$ implies then: %Then, by the Lipschitz   the the Lipschitz assumption \eqref{lipf} implies:
\[\|f(s,\xx(s),u,v^*)-f(t_0,x_0,u,v^*)\| \leq c\big((s-t_0)+\|f\|(s-t_0)\big)=C(s).\]
 From Cauchy-Schwartz inequality and the optimality of $v^*$ one gets:
% By assumption, $(t,x)\mapsto f(t,x,u,v)$ is uniformly Lipschitz. Using 
%$\forall (u,v) \in U\times V$:
\begin{eqnarray}\langle x_0-w_0,f(s,\xx(s),u,v^*)\rangle &\leq&
% \langle \xi, f(t,x,u,v)\rangle +\|\xi\| \|f(s,z,u,v)-f(t,x,u,v)\|\\ 
%&\leq & %\quad \quad \quad  \langle \xi, 
\langle x_0-w_0, f(t_0,x_0,u,v^*)\rangle
+C(s)d_0,\\ \label{aa1}
& \leq & 
%\sup_{u\in U}\langle x_0-w_0, f(t_0,x_0,u,v^*)\rangle+c(1+\|f\|)(t-t_0)d_0,\\
H^+(t_0,x_0,x_0-w_0)+C(s)d_0.
\end{eqnarray}
 %Let $\xi_0=x_0-z_0$. 
Similarly, Assumption $1$ implies $\|\ww(s)-x_0\| \leq d_0+(s-t_0)\|f\|$, and then: % Assumption $(1b)$ gives:
%and also:% The Lipschitz assumption \eqref{lipf} implies then:
\[\|f(s,\ww(s),u^*,v)-f(t_0,x_0,u^*,v)\| \leq C(s)+cd_0.\]
 Using Cauchy-Schwartz inequality, and the optimality of $u^*$:
\begin{eqnarray}\langle x_0-w_0,f(s,\ww(s),u^*,v)\rangle &\geq&
\langle x_0-w_0, f(t_0,x_0,u^*,v)\rangle
-(C(s)+cd_0)d_0, \quad\\ \label{aa2}
& \geq & 
%\sup_{u\in U}\langle x_0-w_0, f(t_0,x_0,u,v^*)\rangle+c(1+\|f\|)(t-t_0)d_0,\\
H^-(t_0,x_0,x_0-w_0)-C(s)d_0-cd_0^2.
\end{eqnarray}
The claim follows: substract the inequalities \eqref{aa1} anc \eqref{aa2} %the two inequalities 
and use Assumption $2$ to cancel $(H^+-H^-)(t_0,x_0,x_0-w_0)$. 
\\
In particular, it holds for $(u,v)=(\uuu(s),\vvv(s))$. Note that $\int_{t_0}^t 2C(s)ds=(t-t_0)C(t)$. Thus, integrating \eqref{seasy} over $[t_0,t]$ yields:
\begin{equation}
\label{eez}
\int_{t_0}^t \langle x_0-w_0,f(s,\xx(s),\uuu(s),v^*)-f(s,\ww(s),u^*,\vvv(s))ds\rangle\leq (t-t_0)(C(t)d_0+cd_0^2).
\end{equation}
% 
% \begin{eqnarray}\label{eez}
% \int_{t_0}^t \langle x_0-w_0,f(s,\xx(s),\uuu(s),v^*)-f(s,\ww(s),u^*,\vvv(s))ds\rangle&\leq& (t-t_0)(C(t)d_0+cd_0^2). \quad
%  \quad \quad
% %&\leq & (2M+c)d_m^2+2M \theta_{m+1}.
% \end{eqnarray}
Using the estimates \eqref{easy} and  \eqref{eez} in \eqref{pr} we obtain: %, and the inequality $d_0\leq 1+ d_0^2$:
% yields: %, we obatin: %, $(t-t_0)C(t)=(t-t_0)^2 c(1+\|f\|)$ and $t-t_0\leq 1$ yield:
\begin{equation*}\dd^2(t) \leq d_0^2 + 4\|f\|^2(t-t_0)^2 + 2(t-t_0)C(t)d_0+2c(t-t_0)d_0^2.
%\\ &=& \big(1+ 2(t-t_0)C(t)+c(t-t_0)\big)d_0^2+ \big(4\|f\|^2 + \big) (t-t_0)^2
\end{equation*}
Finally, use the relations $d_0\leq 1+d_0^2$, $C(t)\leq c(1+\|f\|)$ and $(t-t_0)C(t)=c(1+\|f\|)(t-t_0)^2$. The result follows with $A=3c+2\|f\|$ and $B=4\|f\|^2+2c(1+\|f\|)$.
\end{proof}
%\begin{remarque}
%, and that the bound simplifies  to $\|\xx(t)-\ww(t)\|^2 \leq B(t-t_0)^2$.
%\end{remarque}

\subsection{Consequences}
%Lemma \ref{fundamental} gives a bound for the distance between $\xx(t)$ and $\ww(t)$ .
In this section, we give three direct consequences of Lemma \ref{fundamental}. In Section \ref{it}, we use the set of times $\Pi$
to construct two trajectories on $[t_0,t_N]$ inductively. Applying the lemma to the intervals $[t_m,t_{m+1}]$, from $m=0$ to $N-1$, we obtain a bound for the distance
between the two at time $t_N$. In particular, the distance vanishes as $\|\Pi\|$ and 
%the initial distance 
$\|x_0-w_0\|$ tend to $0$.
In Section \ref{sets}, we replace the distance between two trajectories by the distance between a trajectory and a set. 
Finally, we combine the two aspects in Section \ref{combi}; the result obtained therein is used in Section \ref{jd}
to prove the existence of the value in zero-sum differential games with terminal payoff.
%n Section \ref{it}, 
%For any $\mathcal{Z}\subset \RR^n$, let $D(x,\mathcal{Z}):=\inf_{z\in \mathcal{Z}} \|%x-z\|$ be the usual distance to the set $\Z$.
%, so that Assumptions $1$ and $2$ are hold. 
% First, we apply the result, 
% inductively, to a sequence of times $t_0<\dots <t_N$. Second, we define a property for sets in $[0,1]\times \RR^n$, 
% and apply Lemma \ref{fundamental} to bound the distance to any set having this property.
% Assumptions $1$ and $2$ are supposed to hold.
%\subsubsection{Discretizations}\label{disc}
\subsubsection{Induction}\label{it}%Iterating the key Lemma} Iterating Lemma \ref{fundamental}
%Consider a sequence of times $\Pi=\{t_0<t_1<\dots < t_N \}$ in $[0,1]$, and let $\|\Pi\|:=\max \{t_m-t_{m-1}, m=1,\dots, N\}$. 
Let $(\uuu,\vvv)\in \U\times \V$ be a pair of controls. Define the trajectories $\xx$ and $\ww$ on $[t_0,t_N]$ inductively: let 
$\xx(t_0)=x_0$ and $\ww(t_0)=w_0$ and suppose that $\xx(t)$ and $\ww(t)$ are defined on $[t_0,t_m]$ for some $m=à,\dots,N-1$;
let $(u^*_m,v_m^*)\in U\times V$ be a couple of optimal actions in the local game $\Ga(t_m,\xx(t_m),\xx(t_m)-\ww(t_m))$; 
for all $t\in[t_m,t_{m+1}]$, put $\xx(t):= \xx[t_m,\xx(t_m),\uuu,v_m^*](t)$ and $\ww(t):=\ww[t_m,\ww(t_m),u_m^*,\vvv](t)$.
% \begin {itemize}
%  \item[-] Suppose $\xx(t)$ and $\ww(t)$ are defined on $[t_0,t_m]$, with $\xx(t_0)=x_0$ and $\ww(t_0)=w_0$. 
%  \item[-] Let $(u^*_m,v_m^*)\in U\times V$ be a couple of optimal actions in the local game $\Ga(t_m,\xx(t_m),\xx(t_m)-\ww(t_m))$.
%  \item[-] For all $t\in[t_m,t_{m+1}]$, put $\xx(t):= \xx[t_0,x_0,\uuu,v_m^*](t)$ and $\ww(t):=\ww[t_0,w_0,u_m^*,\vvv](t)$.
%  \end {itemize}

%be the unique absolutely continuous 
%solutions of:
%\begin{eqnarray*}
% \dot{\xx}(t)&=& f(t,\xx(t),\uuu(t),v_m^*),\\
% \dot{\ww}(t)&=& f(t,\ww(t),u_m^*,\vvv(t)).
%\end{eqnarray*}
%$\|\Pi\|=\max\{t_m-t_{m-1},\ m=1,\dots, N\}$ be its mesh.
\begin{corollaire}\label{coro1}% The following estimate holds:
%Under Assumptions $1$ and $2$:
%There exists constants $A,B\geq 0$ such that $\forall t\in [t_0,1]$:
$\|\xx(t_N)-\ww(t_N)\|^2 \leq e^{A}( \|x_0-w_0\|^2+B \|\Pi\|).$
%$\|\xx(t_N)-\ww(t_N)\|^2 \leq e^{A(t_N-t_0)}( \|x_0-w_0\|^2+B (t_N-t_0)\|\Pi\|)$.
\end{corollaire}
\begin{proof}
For any $0\leq m \leq N$, put $d_m:=\|\xx(t_m)-\ww(t_m)\|$. 
% Using $(u_m^*,v_m^*)$ optimal in $\Ga(t_m,\xx(t_m),\xx(t_m)-\ww(t_m))$, one
By Lemma \ref{fundamental}, one has: %, for any $1\leq m\leq N$:
\begin{equation*}
d_m^2\leq (1+(t_m-t_{m-1})A )d_{m-1}^2+B(t_m-t_{m-1})^2.
\end{equation*}%Let $d_m=\|\xx(t_m)-\ww(t_m)\|$. For any $\leq m\lea 
By induction, one obtains:
% Thus, by induction: 
$$d^2_N\leq \exp(A\sum\nolimits_{m=1}^N t_m-t_{m-1})(d_0^2+B\sum\nolimits_{m=1}^N (t_m-t_{m-1})^2).$$
The result follows using that $\sum_{m=1}^N t_m-t_{m-1}\leq 1$ and $\sum\nolimits_{m=1}^N (t_m-t_{m-1})^2 \leq \|\Pi\|$.
%Thus, $d_N^2\leq e^{A(t_N-t_0)}( \|x_0-w_0\|^2+B (t_N-t_0)\|\Pi\|)$. The result follows from $t_N-t_0\leq 1$. %, because the times are in $[0,1]$.
%\leq e^A(d_0^2+B\|\Pi\|).$$
\end{proof}
%\subsubsection{Viability property}
\subsubsection{Distance to a set}% The distance to a set} %Applying Lemma \ref{fundamental} to a set
\label{sets} %pplying Lemma \ref{} to a set}
Let  $\cW \subset [t_0,1]\times \RR^n$ be a set satisfying the following properties:
 \begin{itemize}
 \item[$\bullet \ \textbf{P1:}$] For any $t\in[t_0,1]$, $\cW(t):=\{x\in \RR^n\ | \ (t,x)\in \cW\}$ is closed and nonempty.
 \item[$\bullet \ \textbf{P2:}$] For any $(t,x)\in \cW$ and any $t'\in [t,1]$:
% Let $\mathcal{C}\subset [0,1]\times \RR^n$ and let, for any $t\in [0,1]$, $\mathcal{C}(t):=\{ x\in \RR^n \ | \ (t,x)\in \mathcal{C}\}$ be its cross-section, and 
% let $d$ be the usual distance in $\RR^n$.\\ %Let $d$ denpteFor any set $S\subset \RR^n$, let $d(Notation: let $d:\RR^n\times \mathcal{P}
\begin{equation*}\label{stab}\adjustlimits \sup_{u\in U} \inf_{\vvv\in \V} D(\xx[t,x,u,\vvv](t'),\cW(t'))=0.
 \end{equation*}
%where $d$ is the usual distance in $\RR^n$.
\end{itemize}
The first property ensures that the projection on $\cW(t)$ is well defined for all $t\in [0,1]$. 
%, as a set- valued map, for all $t\in[0,1]$. 
%) while for any $\ep>0$, 
Equivalent formulations of the second property were introduced by Aubin \cite{aubin89victory}, although the formulation 
$\textbf{P2}$ is inspired by the notion of stable bridge in \cite{KS87}.
%%%%%%%%%% CITATION %%%%%%%%%%%
%\noindent \textbf{P2} 
%Property $(P2)$ is closely related to viability, as introduced by Aubin \cite{Aubin}.
%%%%%%%%%% CITATION %%%%%%%%%%%
%%%%%%%%%%%% version 13 %%%%%%%%%
% \begin{corollaire}\label{coro2} Let $\cW\subset[t_0,1]\times \RR^n$ satisfy \emph{\textbf{P1}} and \emph{\textbf{P2}}. Under
% Assumptions $1$ and $2$ there exists $v^*\in V$ such that, for all $x_0\in \RR^n$, $t\in[t_0,1]$ and $\uuu\in \U$,
% %for all $t\in [t_0,1]$ and for all $\uuu\in \U$:
% \begin{equation*}
%  d^2(\xx[t_0,x_0,\uuu,v^*](t),\cW(t))\leq (1+(t-t_0)A)d^2(x_0,\cW(t_0)) + B(t-t_0)^2.
% \end{equation*}
% \end{corollaire}
%%%%%%%%%%%% version 13 %%%%%%%%%

Let $x_0\in \RR^n$, let $w_0\in \mathrm{argmin}_{\cW(t_0)}\|x_0-w_0\|$ be some closest point to $x_0$ in $\cW(t_0)$ and let
$v^*$ be optimal in the local game $\Ga(t_0,x_0,x_0-w_0)$.
%and $t\in[t_0,1]$. % Assume that player 1 plays the control  
\begin{corollaire}\label{coro2}
%Let $\cW\subset[t_0,1]\times \RR^n$ satisfy \emph{\textbf{P1}} and \emph{\textbf{P2}}. Let
%$x_0\in \RR^n$ and $t\in[t_0,1]$. % Assume that player 1 plays the control  
% $\uuu\in \U$ in the interval $[t_0,t]$.
For all $t\in [t_0,1]$ and all $\uuu\in \U$:
%for all $t\in [t_0,1]$ and for all $\uuu\in \U$:
\begin{equation*}
 D^2(\xx[t_0,x_0,\uuu,v^*](t),\cW(t))\leq (1+(t-t_0)A)D^2(x_0,\cW(t_0)) + B(t-t_0)^2.
\end{equation*}
%where $\xx(t):=\xx[t_0,x_0,\uuu,v^*](t)$.
\end{corollaire}
\begin{proof}
%Let $w_0$ be some closest point to $x_0$ in $\cW(t_0)$ and 
Let $\uuu\in \U$ be fixed. %(which exists by \textbf{P1}).
% so that $d(x_0,\cW(0))=\|x_0-w_0\|$. 
Let $u^*$ be optimal in $\Ga(t_0,x_0,x_0-w_0)$.
By \textbf{P2}, for all $\ep>0$ there exists $\vvv_\ep\in \V$ such that the point
$\ww_\ep(t):= \xx[t_0,w_0,u^*,\vvv_\ep](t)$ satisfies
 $D(\ww_\ep(t),\cW(t))\leq \ep$. 
We use the following abbreviation: $\xx_\uuu(t):=\xx[t_0,x_0,\uuu,v^*](t)$. The triangular inequality gives $D(\xx_\uuu(t),\cW(t))\leq \|\xx_\uuu(t)-\ww_\ep(t)
 \|+\ep$. Taking the limit, as $\ep\to 0$, one has that:
\[D^2(\xx_\uuu(t),\cW(t))\leq \lim_{\ep\to0}\|\xx_\uuu(t)-\ww_\ep(t)\|^2.\]% \leq (1+(t-t_0)A)\|x_0-w_0\|^2 + B(t-t_0)^2.\] 
By Lemma \ref{fundamental}, $\|\xx_\uuu(t)-\ww_\ep(t)\|^2\leq (1+(t-t_0)A)\|x_0-w_0\|^2 + B(t-t_0)^2$ for all $\ep>0$. The result follows because $\|x_0-w_0\|=D(x_0,\cW(t_0))$ by definition.  %The result follows.
\end{proof}
%\begin{remarque} If $x_0\in \cW(t_0)$, $w_0=x_0$ and any $v\in V$ is optimal in $\Ga(t_0,x_0,x_0-w_0)$. % and the bound simplifies to
%\end{remarque}

\subsubsection{A key Corollary} \label{combi}
%For any $v_0,\dots,v_{N-1}\in V$, let $(v_m)_m\in \V$ stand for the control which is equal to $v_m$ on $[t_m,t_{m+1}]$.
%Let $\U_0=\emptyset$ and, 
% For any $m=1,\dots,N$ and $\uuu\in \U$, let $\uuu_m:=\uuu_{|[t_0,t_m]}$ be its restriction to the interval $[t_0,t_m]$ and 
% let $\U_m=\{\uuu_m\ | \ \uuu\in \U\}$. 
Let $x_0\in \cW(t_0)$. %, so that $x_0-w_0=0$ and, as a consequence, every $v\in V$ is optimal in the local game $\Ga(t_0,x_0,x_0-w_0)$.
%Let $v_0^*$ be any of them. 
For any $\uuu\in\U$, define a trajectory $\xx_\uuu$ on $[t_0,t_N]$ inductively: let $\xx_\uuu(t_0)=x_0$ and 
suppose that $\xx_\uuu$ is defined on $[t_0,t_m]$ for some $m=0,\dots,N-1$.  
Let $w_m\in \mathrm{argmin}_{w\in \cW(t_m)}\|\xx_\uuu(t_m)-w\|$ be a closest point to $\xx_\uuu(t_m)$ in $\cW(t_m)$, and let
$v_m^*$ be optimal in the local game $\Ga(t_m,\xx_\uuu(t_m),\xx_\uuu(t_m)-w_m)$.\footnote{We implicitly use two selection rules $\pi_1=\pi_1(\cW)$ and $\pi_2$ defined as follows: $\pi_1:[0,1]\times \RR^n\to \RR^n$ assigns to each $(t,x)$ a closest point to $x$ in $\cW(t)$;  $\pi_2:[0,1]\times \RR^n\times \RR^n\to V$ assigns to each $(t,x,\xi)$ an optimal action in the local game $\Ga(t,x,\xi)$. }
%Let $\pi^\phi:[0,1]\times \RR^n\to V$, $(t,x)\mapsto \pi_2(t,x,\pi_1(t,x))$.}
%%%%%%%%%%%%
%%%\footnote{We implicitly use some selection rule $\pi:[0,1]\times \RR^n\to V$ that assigns, to each time and position $(t,x)$, 
%%%an optimal action in the local game $\Ga(t,x,x-w)$, where $w$ is a (previously selected)
%%%closest point to $x$ in $\cW(t)$.} %, when 
%%%%%%%%%%%%
For all $t\in[t_m,t_{m+1}]$, put $\xx_\uuu(t):=\xx[t_m,\xx_\uuu(t_m),\uuu,v_m^*](t)$.
Define a control $\be(\uuu) \in \V$ by setting $\be(\uuu)\equiv v_m^*$ on $[t_m,t_{m+1}]$ for all $0\leq m<N$. 
Clearly, $\xx_\uuu(t)=\xx[t_0,x_0,\uuu,\be(\uuu)](t)$, for all $t\in [t_0,t_N]$.

%Notice that $\xx_\uuu(t_m)\in \cW(t_m)$ implies that $\xx_\uuu(t_m)=w_m$, so that any $v\in V$ is optimal in the local game
%$\Ga(t_m,\xx_\uuu(t_m),0)$. 
Note that the action $v_m^*$ used in the interval $[t_m,t_{m+1}]$ depends only on
the current position $\xx_\uuu(t_m)$ and on the set $\cW(t_m)$, and that the former is a deterministic function of $v_0^*,\dots,v_{m-1}^*$ and of the restriction of $\uuu$ to the interval $[t_0,t_m]$. 
%%%%%%%%%%%%%
%\begin{remarque}\label{selec}Note also that $\beta$ uses implicitly some selection  rule $\pi:[0,1]\times \RR^n\to \RR^n\times V$ to determine, for each time and position $(t,x)$, some $w$ which is a closest point to $x$ in $\cW(t)$, and some optimal action in the local game $\Ga(t,x,x-w)$. %, when there is more than one.
%\end{remarque}
%%%%%%%%%%%%%
%$\Ga(t_m,\xx_\uuu(t_m),\xx_\uuu(t_m)-w_m)$ (when there is more than one) will be played. 
%is played local game at time $t_m$, for $m=0,\dots,N-1$.

Consider $\be$ as a mapping from $\U$ to $\V$. 
Then, for any $\uuu_1, \uuu_2\in \U$ such that 
$\uuu_1\equiv \uuu_2$ on $[t_0,t_m]$ for some $0\leq m<N$,
$\be(\uuu_1)\equiv \be(\uuu_2)$ on $[t_0,t_{m+1}]$. In this sense, $\be:\U\to \V$
is nonanticipative. %This notion will be used in Section \ref{jd}.

Putting Corollaries \ref{coro1} and \ref{coro2} together, and using that $D(x_0,\cW(t_0))=0$,
we obtain the following result.
%%%%%%%%%%% version 13 %%%%%%%%%
% \begin{corollaire}\label{coro3} Let $\cW\subset[t_0,1]\times \RR^n$ satisfy \emph{\textbf{P1}} and \emph{\textbf{P2}},
% and let $x_0\in \cW(t_0)$. Under Assumptions $1$ and $2$ there exist $v_0^*,\dots, v^*_{N-1} \in V$ such that
% for all $\uuu\in \U$: 
% \begin{equation*}
%  d^2(\xx[t_0,x_0,\uuu,(v_m^*)_m](t_N),\cW(t_N))\leq e^A B \|\Pi\|.
% \end{equation*}
% \end{corollaire}
%%%%%%%%%%% version 13 %%%%%%%%%
%$v^*_m=\Phi(\uuu_{|[t_0,t_m]}\in \V$ a sequence  $v_0^*,\dots, v^*_{N-1} \in V$
\begin{corollaire}\label{coro3} For any $\uuu\in \U$, 
%Let $\cW\subset[t_0,1]\times \RR^n$ satisfy \emph{\textbf{P1}} and \emph{\textbf{P2}} and let $x_0\in \cW(t_0)$. % and let $v_0\i. 
%Assume the control functions $\uuu_i: [t_i,t_{i+1})\rightarrow U$ are played by player 1 in their intervals of definition.
%For $m=0,\dots, N-1$, there exists maps $\Phi_m:\U_m\to V$ such that, for all $\uuu\in \U$, $(\Phi_m(\uuu_m))_m$ is such: % that if $\uuu(t):=\sum^{N-1}_{i=0}\uuu_i(t)\ind_{[t_i,t_{i+1})}$: 
%\begin{equation*}
$ D^2(\xx[t_0,x_0,\uuu,\be(\uuu)](t_N),\cW(t_N))\leq e^A B \|\Pi\|.$
%\end{equation*}
\end{corollaire}
%%%%%%%%%%%% version 13 %%%%%%%%%%
%The maps $\Phi_m$ are constructed recursively. $\Phi_0:\emptyset$
This result can be interpreted as follows: suppose the state belongs to $\cW$ at time $t_0$; 
then for any control $\uuu\in \U$ (say, of player $1$), there
exists a piece-wise, nonanticipative reply of player $2$ such that the distance to $\cW$ at some terminal time vanishes as $\|\Pi\|$ tends to $0$.
In terms of strategies, which will be
defined in the next Section, the result implies that player $2$ has a strategy which keeps the state very close to the set $\cW$.
This property is used in Section \ref{jd} to prove the existence of the value in zero-sum differential game.
The epigraph of the lower value function will play the role of the set $\cW$.
%%%%%%%%%%%%%%%%%% version 16 % % 
% % This result can be interpreted as follows: suppose the state belongs to $\cW$ at time $t_0$; 
% % player 2 observes the current state at the times in $\Pi$  and reacts to it such that the distance to $\cW$ at some terminal time vanishes as $\|\Pi\|$ tends to $0$.
% % In terms of strategies, which will be
% % defined in the next Section, the result implies that player $2$ has a strategy which keeps the state very close to the set $\cW$.
% % This property is used in Section \ref{jd} to prove the existence of the value in zero-sum differential game.
% % The epigraph of the lower value function will play the role of the set $\cW$ here.
%%%%%%%%%%%%%%%%%%%%% version 16 %%%%%%%%%%
%This result can be interpreted as follows: suppose the state belongs to $\cW$ at time $t_0$; then, after observing the control $\uuu_0\in U$ of player 1 in the interval $[t_0,t_1)$, player $2$, who remembers his control $v^*_0$, can compute the new state $x_1$ and play optimally in $\Ga(t_1,x_1,x_1-w_1)$, and so on.  

% Corollary \ref{coro3} is used in Section \ref{jd}: the set $\cW$ will be the epigraph 
% of the lower value function of zero-sum differential game. 

\section{Differential Games}\label{jd}
For any $(t_0,x_0)\in [0,1]\times \RR^n$, consider now the zero-sum differential game played in $[t_0,1]$
with the following dynamics: 
\[\xx(t_0)=x_0, \ \text{ and } \ \dot{\xx}(t)=f(t,\xx(t),\uuu(t),\vvv(t)), \text{ a.e. on } [t_0,1].\]
%%%%%%%%%%%%%%%%%%%%%%%%%%%%%%%%%%%%%%%%%%
\begin{definition} A strategy for player $2$ is a map $\be:\U\to \V$ such that,
for some finite partition $s_0<s_1<\cdots < s_N$ of $[t_0,1]$, for all $\uuu_1,\uuu_2 \in\U$ and $0\leq m<N$:
\[\uuu_1\equiv \uuu_2 \text{ a.e. on } [s_0,s_m] \ \Longrightarrow \ \be(\uuu_1)\equiv \be(\uuu_2)\text{ a.e. on }[s_0,s_{m+1}].\]
\end{definition} 
\noindent These strategies are called nonanticipative strategies with delay (NAD) in \cite{CQ08}, in contrast to the classical nonanticipative strategies.
The strategies for player $1$ are defined in a dual manner. Let $\A$ (resp.  $\B$) the set of
strategies for Player $1$ (resp. $2$). For any pair of strategies $(\al,\beta)\in \A\times \B$, there exists a unique pair $(\bar{\uuu},\bar{\vvv})\in \U\times \V$ such that $\al(\bar{\vvv})=\bar{\uuu}$,  and $\be(\bar{\uuu})=\bar{\vvv}$ (see \cite{CQ08}).
This fact is crucial for it allows to define $\xx[t_0,x_0,\al,\be]:=\xx[t_0,x_0,\bar{\uuu},\bar{\vvv}]$ in a unique manner.

%%%%%%%%%%%%%%%%%%%%%%%%%%%%%%%%%%%%%%%%%%

The payoff in a differential game has generally two components: a running payoff and a terminal payoff, represented by the functions $\ga:[0,1]\times \RR^n\times U\times V\to \RR$ and $g:\RR^n\to \RR$ respectively.
However, the classical transformation of a Bolza problem into a Mayer problem, which gets rid of the running payoff, can also be applied here: enlarge the state space from $\RR^n$ to $\RR^{n+1}$, where the last coordinate represents the cumulated payoff; 
define an auxiliary terminal payoff function $\widetilde{g}:\RR^{n+1}\to \RR$ as $\widetilde{g}(x,y)=g(x)+y$; we thus obtain an equivalent 
differential game with no running payoff. W.l.o.g. we assume from now on that $\ga\equiv 0$.
\textbf{Assumption 3:} $g$ is Lipschitz continuous.\\
We suppose that the Assumptions $1$, $2$ and $3$ hold in the rest of the paper. % are 
The differential game with initial time $t_0$, initial state $x_0$ and terminal payoff $g$ is denoted by $\mathcal{G}(t_0,x_0)$.\\
Introduce the lower and upper value functions: % are defined respectively as follows
\begin{eqnarray*}
 V^-(t_0,x_0)&:=&\adjustlimits \sup_{\al \in \A}\inf_{\be\in\B} g(\xx[t_0,x_0,\al,\be](1)),\\ 
 V^+(t_0,x_0)&:=&\adjustlimits \inf_{\beta \in \B}\sup_{\al\in \A} g(\xx[t_0,x_0,\al,\be](1)).  %\right\} \\ \left\{ \yy[t_0,x_0,\al,\be](1)+ \right\}.
\end{eqnarray*}
The inequality $V^-\leq V^+$ holds everywhere. 
If $V^-(t_0,x_0)= V^+(t_0,x_0)$, the game $\mathcal{G}(t_0,x_0)$ has a value. 
Notice that its lower and upper Hamiltonian of are precisely the maxmin and the minmax of the local games defined in Section \ref{dir}.
Consequently, {Assumption $2$} is precisely \defn{Isaacs' condition}. \\ 
The lower value function satisfies the following super-dynamic programming principle (see \cite{CQ08}).

For all $(t,x)\in [t_0,1]\times \RR^n$ and all $t'\in[t_0,1]$:
\begin{equation}\label{sdpp}V^-(t,x)\geq\adjustlimits \sup_{\al\in \A} \inf_{\beta\in \B}V^-(t',\xx[t_0,x_0,\al,\beta](t').
\end{equation}

A proof of a slightly weaker version of \eqref{sdpp} can be found in the Appendix.
\subsection{Existence and characterization of the value}\label{ex}
Let $\phi:[t_0,1]\times \RR^n \to \RR$ be a real function satisfying the following properties:
\begin{itemize}
\item [$(i)$] $x\mapsto \phi(t,x)$ is lower
semicontinuous, for all $t\in [t_0,1]$;
\item[$(ii)$]  %$\phi$ satisfies the following super-dynamic programming principle:
For all $(t,x)\in [t_0,1]\times \RR^n$ and $t'\in [t,1]$: $$\phi(t,x)\geq \adjustlimits \sup_{u\in U}\inf_{\vvv\in
\V}\phi\big(t',\xx[t,x,u,\vvv](t')\big);$$
% \begin{equation}\label{pd}\phi(t_0,z_0)\geq \sup_{u\in U}\inf_{\vvv\in
% \V}\phi\big(t,\zz[t_0,z_0,u,\vvv](t)\big).\end{equation}
\item[$(iii)$] $\phi(1,x)\geq g(x)$,
 for all $x\in \RR^n$.
 \end{itemize}
%Such a function exists.
\begin{definition} For any $\ell\in \RR$, define the \defn{$\ell$-level set of $\phi$} by: % is a stable bridge:
\begin{eqnarray*}%\label{stb}
\cW^\phi_\ell&=&\{(t,x) \in [t_0,1]\times \RR^n\ | \ \phi(t,x)\leq
\ell\},
\end{eqnarray*}
\end{definition}
\begin{lemme}\label{stable} For any $\ell \geq \phi(t_0,x_0)$, the $\ell$-level set of $\phi$ satisfies \emph{\textbf{P1}} and 
\emph{\textbf{P2}}.
\end{lemme}
\begin{proof} $x_0\in \cW^\phi_\ell(t_0)$ so that $\cW^\phi_\ell(t_0)$ is nonempty. By $(i)$,
$\cW^\phi_\ell(t)$ is a closed set for all $t\in[0,1]$. The property $(ii)$ implies that
for any $t\in [t_0,1]$, $u\in U$ and $n\in \NN^*$ there exists $\vvv_n\in \V$ such that:
\begin{equation}\label{rd}\ell\geq \phi(t_0,x_0)\geq \phi\big(t,\xx[t_0,x_0,u,\vvv_n](t))-\frac{1}{n}.\end{equation}
The boundedness of $f$ implies that $x_n:=\xx[t_0,x_0,u,\vvv_n](t)$ belongs to some compact set.
Consider a subsequence $(x_n)_n$ such that $\limn \phi(t,x_n)=\liminf_{ n \to \infty}\phi(t,x_n)$,
and such that $(x_n)_n$ converges to some $\bar{x}\in \RR^n$. 
Take the limit, as $n \to \infty$, in \eqref{rd}. Using $(i)$ again, we obtain:
\begin{equation*}
%\label{rd2}
\ell\geq \phi(t_0,x_0)\geq \phi\big(t,\bar{x}).\end{equation*}
Consequently, $\bar{x}\in \cW_\ell^\phi(t)\neq \emptyset$ and 
$\inf_{n\in \NN^*} d\big(\xx[t_0,x_0,u,\vvv_n](t),\cW^\phi_\ell\big(t))=0$. The proof of these two properties 
still holds by replacing the initial data (i.e. $(t_0,x_0)$ and $t\in [t_0,1]$) by some $(t,x)\in \cW_\ell^\phi$ and $t'\in[t,1]$.
Thus, $\cW_\ell^\phi$ satisfies \textbf{P1} and \textbf{P2}. % Thus, $\cW^\phi_\ell$ is viable for player $2$.
 %$\forallWe have proved that, and satisfies \eqref{stab}.

\end{proof}
\subsubsection{Extremal strategies in $\mathcal{G}(t_0,x_0)$}  %\label{extre}% in differential games}
%Let $\phi$ be some function satisfying $(i)$ and $(ii)$, 
%Let $\Pi=\{t_0<\dots<t_N=1\}$ be partition of $[t_0,1]$,  let $\|\Pi\|=\max\{ t_m-t_{m-1},\ m=1,\dots,N\}$, % be such that $t_m-t_{m-1}\leq \theta$, $\forall m=1,\dots,N$, and let  % of mesh $\theta$.
Let $\cW^\phi\subset[t_0,1]\times \RR^n$ be the $\phi(t_0,x_0)$-level set of $\phi$.
Let $\pi_1=\pi_1(\cW^\phi)$ and $\pi_2$ be two selection rules defined as follows: $\pi_1:[0,1]\times \RR^n\to \RR^n$ assigns to each $(t,x)$ a closest point to $x$ in $\cW^\phi(t)$;  $\pi_2:[0,1]\times \RR^n\times \RR^n\to V$ assigns to each $(t,x,\xi)$ an optimal action in the local game $\Ga(t,x,\xi)$. Finally, let:
% $\pi:[0,1]\times \RR^n\to V$, $(t,x)\mapsto \pi_2(t,x,x-\pi_1(t,x))$.
\[ \pi:[0,1]\times \RR^n \to  V,\quad (t,x)\mapsto \pi_2(t,x,x-\pi_1(t,x)).\]
%\begin{eqnarray} \pi:[0,1]\times \RR^n &\to & V,\\(t,x)&\mapsto& \pi_2(t,x,\pi_1(t,x)).\end{eqnarray}
%%%%%%%%%%%%%%%%
%Let $\pi:[0,1]\times \RR^n\to V$ be a selection rule that assigns, to each time and position $(t,x)$,  an optimal action in the local game $\Ga(t,x,x-w)$, where $w$ is a (previously selected) closest point to $x$ in $\cW(t)$. %, when 
 %%%%%%%%%%%%%%%%
% \begin{definition} An \defn{extremal strategy} $\be=\be(\phi,\Pi)$ is defined inductively as follows: suppose $\be$ is already defined in $[t_0,t_m]$ and let $x_m=\xx[t_0,x_0,\uuu,\be](t_m)$. Then, for $\uuu\in \U$:
% \begin{itemize} 
% \item If $x_m\in \cW^\phi(t_{m})$, set $\be(\uuu)(s)=v$, for some $v\in V$, for all $s \in [t_m,t_{m+1})$.
% \item If $x_m\notin \cW^\phi(t_{m})$, let $w_m\in \mathrm{argmin}_{w\in \cW^\phi(t_m)} \|x_m-w_m\|$ be some closest point to $x_m$ in $\cW(t_m)$, and 
% let $v^*_m$ be an optimal action in the local game $\Ga(t_m,x_m,x_m-w_m)$.
% Set $\be(\uuu)(s)=v^*_m$, for all $s \in [t_m,t_{m+1})$.
% \end{itemize}
%  \end{definition}
%%%%%%%%%%%%% version clasica %%%%%%%%%%%%%%%%
 %%%%%%%%%%%%%%%%%%%%%%%%% ALternaztive definition in one case %%%%%%%%%%%%%
\begin{definition} An \defn{extremal strategy} $\be=\be(\phi,\Pi,\pi):\U\to \V$ is defined inductively as follows:
suppose that $\be$ is already defined in $[t_0,t_m]$ for some $0\leq m <N$, and let $x_m=\xx[t_0,x_0,\uuu,\be](t_m)$. Set 
%Let $v^*_m:=\pi(t_m,\xx[t_0,x_0,\uuu,\be](t_m))$. Then
% be an optimal action in $\Ga(t_m,x_m,x_m-w_m)$, 
%where $w_m$ is a (previously selected) closest point to $x_m$ in $\cW^\phi(t_m)$.  %and some optimal action in. %\footnote{ See footnote in page 4.}.
% (see Remark \ref{selec}).
 $\be(\uuu)\equiv \pi(t_m,x_m)$ on $[t_m,t_{m+1}]$.
\end{definition}
These strategies are inspired by the \emph{extremal aiming} method of Krasovskii and Subbotin (see Section 2.4 in \cite{KS87}).

% Whenever $x_m\notin \cW(t_m)$,
% player $2$ plays optimally in the \emph{instantaneous game} with direction $x_m-z_m$, i.e. a one-shot game with strategy sets $U$ and $V$, and
%  payoff function $r(u,v)=\langle x_m-z_m, f(t_{m},x_m,u,v)\rangle$, where $x_m-z_m$ is a proximal normal vector to $\cW(t_m)$ at $z_m$.
\begin{proposition}\label{cc} 
%For any $\phi$ satisfying $(i)$, $(ii)$ and $(iii)$ $\phi(1,x)=g(x)$,
% $\forall x\in \RR^n$, and 
%Under Assumptions $1$, $2$ and $3$, t
For some $C\geq 0$, and for any extremal strategy $\be=\be(\phi,\Pi,\pi)$: %, for all $\uuu\in\U$: % $\forall \theta>0$, $\exists \be\in \B_\theta$ such that:
\[g(\xx[t_0,x_0,\uuu,\be(\uuu)](1))\leq \phi(t_0,x_0)+ C \sqrt{\|\Pi\|}, \quad \forall \uuu\in\U.\]
% is an extremal strategy.
\end{proposition}
\begin{proof} Recall that  $x_N=\xx[t_0,x_0,\uuu,\be(\uuu)](1)$. By Lemma \ref{stable}, 
$\cW^\phi$ satsifies \textbf{P1} and \textbf{P2}.
Thus, by Corollary \ref{coro3}:% implies, since $t_N=1$:
%  Use Assumptions $1$ and $2$ to apply Corollaries \ref{coro1} and \ref{coro2}.
%Put together, they imply: %1, and use the equalities $t_N=1$ and $d_0=0$ to obtain:
\begin{equation}\label{est}D^2(x_N,\cW^\phi(t_N)) \leq e^{A}B\|\Pi\|.  %d_N^2 \leq e^A B,
\end{equation}
% 
% \[d^2(x_N,\cW^\phi(t_N)) \leq e^{A}\big(d^2(x_0,\cW^\phi(t_0))+ B\|\Pi\|\big).  %d_N^2 \leq e^A B,
% \] 
% However, $x_0\in \cW^\phi(t_0)$, so that $d(x_0,\cW^\phi(t_0))=0$.
% and $t_N=1$. Thus, simplifying, $d^2(x_N,\cW^\phi(t_N)) \leq e^A B \|\Pi\|$.
% \begin{equation}\label{r}d(x_N,\cW^\phi(t_N))^2 \leq e^A B \theta.
% \end{equation}
Using $(iii)$ and the fact that $t_N=1$ yields:
%On the other hand, $t_N=1$ so that $(iii)$ implies:
$$\cW^\phi(t_N)=\{x\in \RR^n | \ \phi(1,x)\leq \phi(t_0,x_0)\}\subset \{x\in \RR^n | \ g(x)\leq \phi(t_0,x_0)\}.$$
Let $w_N\in \mathrm{argmin}_{w\in \cW^\phi(1)}\|x_N-w\|$ be some closest point to $x_N$ in $\cW^\phi(1)$. 
Let $\kappa$ be the Lipschitz constant of $g$ (assumption $3$). Then:
 \begin{eqnarray*}
 g(x_N)&\leq& g(w_N)+\kappa \|x_N-w_N\|,\\&\leq & \phi(t_0,x_0)+ \kappa d(x_N,\cW^\phi(t_N)).
 %\\  &\leq & \phi(t_0,x_0)+ \kappa \sqrt{e^AB} \sqrt{\|\Pi\|}.
 % d(x_N,\cW^\phi(t_N)),
 \end{eqnarray*}
The result follows from \eqref{est}. %: recall that Explicitly, $C=c\sqrt{e^AB}$.
%\[d\big(\xx[t_0,x_0,\uuu,\be(\uuu)](1), \cW^\phi((1))\leq \sqrt{e^A B\theta}.\]
\end{proof}
\noindent Proposition \ref{cc} applies to any function satisfying $(i)$, $(ii)$ and $(iii)$. Consequently, under Assumptions $1$, $2$ and $3$: %uppose that $f$ satisfies the assumptions in Section \ref{prelim}, that $g$ is  $C$-Lipschitz continuous and that Isaacs' condition holds. Then:
\begin{equation}\label{v+}V^+(t_0,x_0)\leq \inf\{\phi(t_0,x_0)\ | \ \phi:[t_0,1]\times \RR^n\to \RR \text{ satisfying }(i),(ii),(iii)\}.
\end{equation}
Consequently, the value exists if the lower value function $V^-$ satisfies $(i)-(iii)$.
%, which is proved in the following Theorem.
\begin{theoreme}\label{main} The differential game $\mathcal{G}(t_0,x_0)$ has a value, characterized as: 
%/$V(t_0,x_0)$, which satisfies:
\begin{equation*}
%\label{car}
\VV(t_0,x_0)=\inf_{\substack{\phi \text{ satisfying }\\(i),(ii),(iii)}}\phi(t_0,x_0).
\end{equation*}
The strategies $\be(\VV,\Pi)$ are asymptotically optimal for player $2$, as $\|\Pi\|\to 0$. %, in $\Ga(f,g)(t_0,x_0)$.
% it is the smallest function satisfying $(i)$, $(ii)$ and $(iii)$.
\end{theoreme}
\begin{proof}
%%%%%%%%%%%%%%%%%%%%%%%%%%%%%% 
%By \eqref{v+}, it is enough to prove that $V^-$ satisfies $(i)$, $(ii)$ and $(iii)$, where 
By definition, $V^-(1,x)=g(x)$, for all $x\in \RR^n$, so that $(iii)$ is satified. The property $(ii)$ 
can be deduced directly from %is a weak version of 
\eqref{sdpp}; % as we explain in the Appendix. 
Assumption $1$ and $3$ imply that the map $x\mapsto V^-(t,x)$ is Lipschitz continuous for all $t\in[t_0,1]$; in particular, $(i)$ is satisfied.
These properties being classical, we have preferred to give the details in the Appendix. %can be found in the Appendix for the sake of completeness.
%Finally, 
Finally, let $\be=\be(\VV,\Pi,\pi)$ be an extremal strategy. %, and let $\uuu\in \U$ be a best reply to $\be$ in $\mathcal{G}(t_0,x_0)$. 
% using Assumption $1$ and Gronwall's lemma one obtains that, for all $t\in[t_0,1]$, $(\uuu,\vvv)\in \U\times \V$, and $ x,y\in \RR^n$:
% $$\left\|\xx[t_0,x,\uuu,\vvv](t)\big)-\xx[t_0,y,\uuu,\vvv](t)\big)\right|\leq e^{c(t-t_0)}\|x-y\|.$$  
% Thus, by Assumption $3$, for all $(\uuu,\vvv)\in \U\times \V$, and for all  $x,y\in \RR^n$:  
% $$\left|g\big(\xx[t_0,x,\uuu,\vvv](1)\big)-g\big(\xx[t_0,y,\uuu,\vvv](1)\big)\right|\leq \kappa e^{c(1-t_0)}\|x-y\|.$$ 
% Consequently, the map $x\mapsto V^-(t,x)$ is $\kappa e^c$-Lipschitz continuous for all $t\in[t_0,1]$. In particular, $(i)$ is satisfied.\\ 
%Let $\be=\be(\VV,\Pi)$ be an extremal strategy. 
Proposition \ref{cc} gives:
$$V^+(t_0,x_0)\leq \sup_{\uuu\in \U} g\big(\xx[t_0,x_0,\uuu,\be(\uuu)](1)\big)\leq V^-(t_0,x_0)+ C\sqrt{\|\Pi\|}.$$ %, \quad \forall \uuu\in \U.$$
The existence of the value is obtained by letting $\|\Pi\|$  tend to $0$. 
Moreover, note that for any $\ep>0$, $\be$ is $\ep$-optimal for sufficiently small $\|\Pi\|$. %$\forall \theta<\theta_0$.
\end{proof}

%%%%%%%%%%%%%%%%%%%%% SUPER SOLUTIONS %%%%%%%%%%%%%%%
% \noindent \textbf{Remark:} The characterization obtained in Theorem \ref{main} is very much related to %a characterization in terms
% viscosity supersolutions of the following Hamilton-Jacobi equation in $[0,1]\times \RR^n$: 
% \begin{equation}\label{HJ}
% \partial_t w(t,x)+H(t,x,D w)=0.
% \end{equation}
% Theorem \ref{main} implies that the value of $\mathcal{G}(t_0,x_0)$,
% is the smallest viscosity supersolution of \ref{HJ} with terminal condition $w(1,x)=g(x)$, $\forall x\in \RR^n$. 
%%%%%%%%%%%%%%%%%%%%% SUPER SOLUTIONS %%%%%%%%%%%%%%%
\section{Appendix} 
Note that the classical subdynamic programming principle \eqref{sdpp} implies $(ii)$.
Indeed, any $u\in U$ can be identified with a strategy that plays $u$ on $[t_0,1]$ regardless of $\vvv$. Then:
%$u:\V\to \U$, $u(\vvv)\equiv u$ on $[t_0,1]$, for all $\vvv\in \V$. Then:
 \begin{eqnarray*}
 %V^-(t,x)\geq
 \adjustlimits \sup_{\al\in \A} \inf_{\beta\in \B}
 V^-(t',\xx[t_0,x_0,\al,\beta])(t') & \geq &\adjustlimits \sup_{u\in U} \inf_{\beta\in \B} V^-(t',\xx[t_0,x_0,u,\beta(u)])(t')  \\
 &=& \sup_{u\in U}\inf_{\vvv \in \V} V^-(t',\xx[t_0,x_0,u,\vvv])(t'). 
 \end{eqnarray*} 
The proofs of \eqref{sdpp} and $(ii)$ are essentially the same. We provide here a proof of the latter 
because it is this version that we have used in the proof of Theorem \ref{main}. % requires this weaker version. %only $(ii)$, which we prove below. 
\begin{claim} $V^-$ satisfies $(i)$ and $(ii)$.
\end{claim}
%However, the proofs are essentially the same in both cases and the proof of Theorem \ref{main} requires only $(ii)$.  
\begin{proof}
%Let us prove that $V^-$ satisfies $(i)$ and $(ii)$.
%The latter is 
$(i)$: Using Assumption $1$ and Gronwall's lemma one obtains that, for all $t\in[t_0,1]$, $(\uuu,\vvv)\in \U\times \V$, and $ x,y\in \RR^n$:
 $$\left\|\xx[t_0,x,\uuu,\vvv](t)\big)-\xx[t_0,y,\uuu,\vvv](t)\big)\right|\leq e^{c(t-t_0)}\|x-y\|.$$  
 Thus, by Assumption $3$, for all $(\uuu,\vvv)\in \U\times \V$, and for all  $x,y\in \RR^n$:  
 $$\left|g\big(\xx[t_0,x,\uuu,\vvv](1)\big)-g\big(\xx[t_0,y,\uuu,\vvv](1)\big)\right|\leq \kappa e^{c(1-t_0)}\|x-y\|.$$ 
 Consequently, the map $x\mapsto V^-(t,x)$ is $\kappa e^c$-Lipschitz continuous for all $t\in[t_0,1]$. \\
 %In particular, $(i)$ is satisfied.
% $(i)$. Let $t\in [t_0,1]$ be fixed. Then,
% $\forall (\uuu,Wv)\in \U\times \V$, and $\forall  x,y\in \RR^n$, 
% Assumption $1$ and by Gronwall's lemma imply: 
% \begin{equation}\label{lip}\|\xx[t,x,\uuu,Wv](1)-\xx[t,y,\uuu,Wv](1)\|\leq \|x-y\|e^{c(1-t)}\leq \|x-y\|e^c,
% \end{equation}
% Assumption $2$ yields: %,  $\forall (\uuu,Wv)\in \U\times \V$, $\forall (x,y)\in (\RR^n)^2$:
% $\left|g\big(\xx[t,x,\uuu,Wv](1)\big)-g\big(\xx[t,y,\uuu,Wv](1)\big)\right|\leq ce^c\|x-y\|$, 
% which implies that $x\mapsto V^-(t,x)$ is $Ce^c$-Lipschitz and, in particular, $V^-$ satisfies $(i)$. On the ther other hand,
% \noindent $(ii)$. 
%Using that $V^-$ satisfies \eqref{sdpp}, we can easily justify that it satisfies $(ii)$. For this, take $\alpha$ a constant strategy (i.e. $\alpha(\vvv)=u$, for all $\vvv $) and any strategy $\beta$. To the couple $(\alpha, \beta)$ associate the pair of controls $(u, \vvv^*) $ such that $\beta(u)=\vvv^*$. Then
% Since this bound is uniform on $u$, it holds for the supremum as well.  \\
% A direct proof that $V^-$ satisfies $(ii)$ goes as follows:
%%%%%%%%%%%%%%%%%%%
%%%% PRUEBA (ii) %%%%%%%%
%%%%%%%%%%%%%%%%%%%
$(ii)$: Let $(t,x)\in [t_0,1]\times \RR^n$, $t'\in [t,1]$ and $\ep>0$ be fixed.
% For any ${x'}\in\RR^n$, let $\al_\ep^{x'}\in\A$ be $\ep$-optimal for 
% player $1$ in $\mathcal{G}(t',{x'})$:
% $$g(\xx[t',x',\al^{x'}_\ep,\be](1)\geq V^-(t',x)-\ep, \ \forall \be\in \B.$$
The Lipschitz continuity of $z\mapsto V^-(t',z)$ implies the existence of some $\de>0$ such that any $\ep$-optimal action in $\mathcal{G}(t',x')$
is $2\ep$-optimal in $\mathcal{G}(t',z)$,
for all $z\in B(x',\de)$, which is the euclidean ball of radius $\de$ centered in $x'$. %\footnote{For any $x\in \RR^n$ and $r>0$, $B(x,r):=\{y\in \RR^n,\ \|y-x\|\leq r\}$.}
%stands for the closed euclidean ball in $\RR^n$, centered in $x$ and of radius $r$.}
%:=\{y\in \RR^n, $ such that $\|z-x'\|<\de$. 
% $\al_\ep^{x'}$ is $2\ep$-optimal in
% $\mathcal{G}(t',y)$, 
% $\forall y\in B({x'},\de)$. 
By compactness, let $B(x,\|f\|)$ be covered by %=\bigcup_{i\in I} E_i$ %B(x_i,\de)$ 
some finite family $(E_i)_{i\in I}$ of pairwise disjoint sets, each one included in a ball $B(x_i,\de)$ for some $(x_i)_i \in (\RR^n)^I$. %, whose radius are smaller . 
%be covered by finitely many balls of radius $\de$.
% $B(z_0,\|f\|)$ be covered by finitely many balls of radius at most $\de$, i.e.  
%By induction, define disjoint, nonempty sets $E_i\subset B(x_i,\de)$ such that $B(x,\|f\|)=\bigcup_{i\in I} E_i$. 
Let $\al_i\in \A$ be an $\ep$-optimal strategy for player $1$ in $V^-(t',x_i)$, and use the notation $\xx_\vvv(t'):=\xx[x,t,u,\vvv](t')$.
Then, by definition, for all $\vvv\in V$:
%, which is $2\ep$-optimal in $\mathcal{G}(t',z)$, for all $z\in E_i$. 
%ThenThus, abbreviating 
%in $\Ga(f,g)(t_1,z_i)$, i.e.
$$g(\xx[t',\xx_\vvv(t'),\al_i,\vvv](1)) \ind_{\{\xx_\vvv(t') \in E_i\}} \geq V^-(t',\xx(t'))\ind_{\{\xx_\vvv(t') \in E_i\}}-2\ep.$$ %, \quad  \forall y\in E_i.$$ .
%For any $Wv\in \V(t_0)$, let $\bar{Wv}:[t_1,1]\to V$ be the restriction of $Wv$ to $[t_1,1]$.
For each $u\in U$, %Let $u\in U$ be $\ep$-optimal in  $\sup_{u\in U}\inf_{Wv\in \V} V^-(t',\xx(t'))$. 
define a strategy $\al_u \in \A$ as follows: $\forall t\in [t_0,1]$, $\forall \vvv\in \V$,
%using the finite collection $(\al_i)_{i\in I}$, 
 \begin{equation*}
\al_u(\vvv)(t)=
\begin{cases} u & \text{if } t\in [t,t'),\\
\al_i(\vvv)(t) &  \text{if } t\in [t',1], \ \text{ and} \quad \xx_\vvv(t')\in E_i.        
          \end{cases}
 \end{equation*}
Note that $\al_u$ is a NAD strategy in $\mathcal{G}(t,x)$. Indeed, let $s_1<\dots<s_N$ be a common partition of $[t',1]$ for the strategies 
$(\al_i)_i$ -- this is possible because the family is finite. Thus, $\al_u$ is defined with respect to $t<t'<s_2<\dots<s_N$.
Now, for all $\vvv\in \V$:%By definition, one has:
\begin{eqnarray*}% and using the Lipschitz continuity of $x\mapsto V^-(t,x)$, 
g(\xx[x,t,\al_u,\vvv](1))&=&\sum\nolimits_{i\in I} g(\xx[t',\xx_\vvv(t'),\al_i,\vvv](1)) \ind_{\{\xx_\vvv(t')\in E_i\}},\\
&\geq& \sum\nolimits_{i\in I} V^-(t',\xx_\vvv(t'))\ind_{\{\xx_\vvv(t')\in E_i\}}-2\ep,\\
&=&V^-(t',\xx_\vvv(t'))-2\ep,
\end{eqnarray*}
Taking the infimum in $\V$, and the supremum in $U$ yields the desired result.

\end{proof}


\begin{thebibliography}{99}

\bibitem{CQ08}
{\sc Cardaliaguet, P. ; Quincampoix, M.} (2008) \textit{Deterministic differential games under probability knowledge of initial condition,} \newblock{International Game Theory Rev.} \textbf{10}, 1--16.

\bibitem{aubin89victory}
{\sc Aubin, J.P.} (1989) \textit{Victory and defeat in differential games}, {Lecture Notes in Control and Inform. Sci.} \textbf{121}, 337--347, Springer, Berlin. 

\bibitem{KS87}
{\sc  Krasovski{\u\i}, N. N. ; Subbotin, A. I.} (1988) \textit{Game-theoretical control problems}, {Springer-Verlag}.

\end{thebibliography}
\end{document}